\documentclass[12pt,a4paper]{article}

\title{Asymptotic Expansions for Moment Functionals of Perturbed Discrete Time Semi-Markov Processes}
\author{Mikael Petersson\footnote{Department of Mathematics, Stockholm University, SE-106 91 Stockholm, Sweden, mikpe@math.su.se.}}
\date{}

\usepackage{amsmath}
\usepackage{amssymb}
\usepackage{amsthm}

\newtheorem{theorem}{Theorem}
\newtheorem{lemma}{Lemma}

\begin{document}

\maketitle

\begin{abstract}
In this paper, we study mixed power-exponential moment functionals of nonlinearly perturbed semi-Markov processes in discrete time. Conditions under which the moment functionals of interest can be expanded in asymptotic power series with respect to the perturbation parameter are given. We show how the coefficients in these expansions can be computed from explicit recursive formulas. In particular, the results of the present paper have applications for studies of quasi-stationary distributions.
\end{abstract}

\noindent \textbf{Keywords:} Semi-Markov process, Perturbation, Asymptotic expansion, Renewal equation, Solidarity property, First hitting time. \\
\\
\textbf{MSC2010:} Primary 60K15; Secondary 41A60, 60K05.

\section{Introduction}

The aim of this paper is to present asymptotic power series expansions for some important moment functionals of non-linearly perturbed semi-Markov processes in discrete time and to show how the coefficients in these expansions can be calculated from explicit recursive formulas. These asymptotic expansions play a fundamental role for the main result in Petersson (2016), which is a sequel of the present paper.

For each $\varepsilon \geq 0$, we let $\xi^{(\varepsilon)}(n)$, $n=0,1,\ldots,$ be a discrete time semi-Markov process on the state space $X = \{ 0,1,\ldots,N \}$. It is assumed that the process $\xi^{(\varepsilon)}(n)$ depends on $\varepsilon$ in the sense that its transition probabilities $Q_{i j}^{(\varepsilon)}(n)$ are continuous at $\varepsilon = 0$ when considered as a function of $\varepsilon$. Thus, we can, for $\varepsilon > 0$, interpret the process $\xi^{(\varepsilon)}(n)$ as a perturbation of $\xi^{(0)}(n)$.

Throughout the paper, we consider the case where the states $\{ 1,\ldots,N \}$ is a communicating class of states for $\varepsilon$ small enough. Transitions to state $0$ may, or may not, be possible both for the perturbed process and the limiting process. It will also be natural to consider state $0$ as an absorbing state but the results hold even if this is not the case.

Our main objects of study are the following mixed power-exponential moment functionals,
\begin{equation} \label{eq:momentfunctionals}
 \phi_{i j}^{(\varepsilon)}(\rho,r) = \sum_{n=0}^\infty n^r e^{\rho n} g_{i j}^{(\varepsilon)}(n), \quad \omega_{i j s}^{(\varepsilon)}(\rho,r) = \sum_{n=0}^\infty n^r e^{\rho n} h_{i j s}^{(\varepsilon)}(n),
\end{equation}
where $\rho \in \mathbb{R}$, $r=0,1,\ldots,$ $i,j,s \in X$,
\begin{equation*}
 g_{i j}^{(\varepsilon)}(n) = \mathsf{P}_i \{ \mu_j^{(\varepsilon)} = n, \ \mu_0^{(\varepsilon)} > \mu_j^{(\varepsilon)} \},
\end{equation*}
\begin{equation*}
 h_{i j s}^{(\varepsilon)}(n) = \mathsf{P}_i \{ \xi^{(\varepsilon)}(n) = s, \ \mu_0^{(\varepsilon)} \wedge \mu_j^{(\varepsilon)} > n \},
\end{equation*}
and $\mu_j^{(\varepsilon)}$ is the first hitting time of state $j$.

As is well known, power moments, exponential moments, and, as in \eqref{eq:momentfunctionals}, a mixture of power and exponential moments, often play important roles in various applications. One reason that the moments defined by equation \eqref{eq:momentfunctionals} is of interest is that the probabilities $P_{i j}^{(\varepsilon)}(n) = \mathsf{P}_i \{ \xi^{(\varepsilon)}(n) = j, \ \mu_0^{(\varepsilon)} > 0 \}$ satisfy the following discrete time renewal equation,
\begin{equation*}
 P_{i j}^{(\varepsilon)}(n) = h_{i i j}^{(\varepsilon)}(n) + \sum_{k=0}^n P_{i j}^{(\varepsilon)}(n-k) g_{i i}^{(\varepsilon)}(n), \ n=0,1,\ldots
\end{equation*}
This can, for example, be used in studies of quasi-stationary distributions as is illustrated in Petersson (2016).

Under the assumption that mixed power-exponential moments for transition probabilities can be expanded in asymptotic power series with respect to the perturbation parameter, we obtain corresponding asymptotic expansions for the moment functionals in equation \eqref{eq:momentfunctionals}. These expansions together with explicit formulas for calculating the coefficients in the expansions are the main results of this paper.

In order to achieve this, we use methods from Gyllenberg and Silvestrov (2008) where corresponding moment functionals for continuous time semi-Markov processes are studied. These methods are based on first deriving recursive systems of linear equations connecting the moments of interest with moments of transition probabilities and then successively build expansions for solutions of such systems.

Analysis of perturbed Markov chains and semi-Markov processes constitutes a large branch of research in applied probability, see, for example, the books by Kartashov (1996), Yin and Zhang (1998), Koroliuk and Limnios (2005), Gyllenberg and Silvestrov (2008), and Avrachenkov, Filar, and Howlett (2013). More detailed comments on this and additional references are given in Petersson (2016).

Let us now briefly outline the structure of the present paper. In Section \ref{sec:semimarkov} we define perturbed discrete time semi-Markov processes and formulate our basic conditions. Then, systems of linear equations for exponential moment functionals are derived in Section \ref{sec:systems} and in Section \ref{sec:convergence} we show convergence for the solutions of these systems. Finally, in Section \ref{sec:expansions}, we present the main results which give asymptotic expansions for mixed power-exponential moment functionals.

\section{Perturbed Semi-Markov Processes} \label{sec:semimarkov}

In this section we define perturbed discrete time semi-Markov processes and formulate some basic conditions.

For every $\varepsilon \geq 0$, let $(\eta_n^{(\varepsilon)}, \kappa_n^{(\varepsilon)})$, $n = 0,1,\ldots,$ be a discrete time Markov renewal process, i.e., a homogeneous Markov chain with state space $X \times \mathbb{N}$, where $X = \{ 0,1,\ldots,N \}$ and $\mathbb{N} = \{ 1,2,\ldots \}$, an initial distribution $Q_i^{(\varepsilon)} = \mathsf{P} \{ \eta_0^{(\varepsilon)} = i \}$, $i \in X$, and transition probabilities which do not depend on the current value of the second component, given by
\begin{equation*}
 Q_{i j}^{(\varepsilon)}(k) = \mathsf{P} \{ \eta_{n+1}^{(\varepsilon)} = j, \ \kappa_{n+1}^{(\varepsilon)} = k \, | \, \eta_n^{(\varepsilon)} = i, \ \kappa_n^{(\varepsilon)} = l \}, \ k,l \in \mathbb{N}, \ i,j \in X.
\end{equation*}
In this case, it is known that $\eta_n^{(\varepsilon)}$, $n=0,1,\ldots,$ is also a Markov chain with state space $X$ and transition probabilities,
\begin{equation*}
 p_{i j}^{(\varepsilon)} = \mathsf{P} \{ \eta_{n+1}^{(\varepsilon)} = j \, | \, \eta_n^{(\varepsilon)} = i \} = \sum_{k=1}^\infty Q_{i j}^{(\varepsilon)}(k), \ i,j \in X.
\end{equation*}

Let us define $\tau^{(\varepsilon)}(0) = 0$ and $\tau^{(\varepsilon)}(n) = \kappa_1^{(\varepsilon)} + \cdots + \kappa_n^{(\varepsilon)}$, for $n \in \mathbb{N}$. Furthermore, for $n=0,1,\ldots,$ we define $\nu^{(\varepsilon)}(n) = \max \{ k : \tau^{(\varepsilon)}(k) \leq n \}$. The discrete time semi-Markov process associated with the Markov renewal process $(\eta_n^{(\varepsilon)}, \kappa_n^{(\varepsilon)})$ is defined by the following relation,
\begin{equation*}
 \xi^{(\varepsilon)}(n) = \eta^{(\varepsilon)}_{\nu^{(\varepsilon)}(n)}, \ n=0,1,\ldots,
\end{equation*}
and we will refer to $Q_{i j}^{(\varepsilon)}(k)$ as the transition probabilities of this process.

In the semi-Markov process defined above, we have that (i) $\kappa_n^{(\varepsilon)}$ are the times between successive moments of jumps, (ii) $\tau^{(\varepsilon)}(n)$ are the moments of the jumps, (iii) $\nu^{(\varepsilon)}(n)$ are the number of jumps in the interval $[0,n]$, and (iv) $\eta_n^{(\varepsilon)}$ is the embedded Markov chain.

It is sometimes convenient to write the transition probabilities of the semi-Markov process as $Q_{i j}^{(\varepsilon)}(k) = p_{i j}^{(\varepsilon)} f_{i j}^{(\varepsilon)}(k)$, where
\begin{equation*}
 f_{i j}^{(\varepsilon)}(k) = \mathsf{P} \{ \kappa_{n+1}^{(\varepsilon)} = k \, | \, \eta_n^{(\varepsilon)} = i, \ \eta_{n+1}^{(\varepsilon)} = j \}, \ k \in \mathbb{N}, \ i,j \in X,
\end{equation*}
are the conditional distributions of transition times.

We now define random variables for first hitting times. For each $j \in X$, let $\nu_j^{(\varepsilon)} = \min \{ n \geq 1 : \eta_n^{(\varepsilon)} = j \}$ and $\mu_j^{(\varepsilon)} = \tau(\nu_j^{(\varepsilon)})$. Then, $\nu_j^{(\varepsilon)}$ and $\mu_j^{(\varepsilon)}$ are the first hitting times of state $j$ for the embedded Markov chain and the semi-Markov process, respectively. Note that the random variables $\nu_j^{(\varepsilon)}$ and $\mu_j^{(\varepsilon)}$, which may be improper, take values in the set $\{ 1,2,\ldots,\infty \}$.

Let us define
\begin{equation*}
 g_{i j}^{(\varepsilon)}(n) = \mathsf{P}_i \{ \mu_j^{(\varepsilon)} = n, \ \nu_0^{(\varepsilon)} > \nu_j^{(\varepsilon)} \}, \ n =0,1,\ldots, \ i,j \in X,
\end{equation*}
and
\begin{equation*}
 g_{i j}^{(\varepsilon)} = \mathsf{P}_i \{ \nu_0^{(\varepsilon)} > \nu_j^{(\varepsilon)} \}, \ i,j \in X.
\end{equation*}
Here, and in what follows, we write $\mathsf{P}_i(A^{(\varepsilon)}) = \mathsf{P} \{ A^{(\varepsilon)} \, | \, \eta_0^{(\varepsilon)} = i \}$ for any event $A^{(\varepsilon)}$. Corresponding notation for conditional expectation will also be used.

Moment generating functions for distributions of first hitting times are defined by
\begin{equation} \label{eq:defphi}
 \phi_{i j}^{(\varepsilon)}(\rho) = \sum_{n=0}^\infty e^{\rho n} g_{i j}^{(\varepsilon)}(n) = \mathsf{E}_i e^{\rho \mu_j^{(\varepsilon)}} \chi( \nu_0^{(\varepsilon)} > \nu_j^{(\varepsilon)}), \ \rho \in \mathbb{R}, \ i,j \in X.
\end{equation}

Furthermore, let us define the following exponential moment functionals for transition probabilities,
\begin{equation*}
 p_{i j}^{(\varepsilon)}(\rho) = \sum_{n=0}^\infty e^{\rho n} Q_{i j}^{(\varepsilon)}(n), \ \rho \in \mathbb{R}, \ i,j \in X,
\end{equation*}
where we define $Q_{i j}^{(\varepsilon)}(0) = 0$.

Let us now introduce the following conditions, which we will refer to frequently throughout the paper:

\begin{enumerate}
\item[$\mathbf{A}$:]
\begin{enumerate}
\item[$\mathbf{(a)}$] $p_{i j}^{(\varepsilon)} \rightarrow p_{i j}^{(0)}$, as $\varepsilon \rightarrow 0$, $i \neq 0$, $j \in X$.
\item[$\mathbf{(b)}$] $f_{i j}^{(\varepsilon)}(n) \rightarrow f_{i j}^{(0)}(n)$, as $\varepsilon \rightarrow 0$, $n \in \mathbb{N}$, $i \neq 0$, $j \in X$.
\end{enumerate}
\end{enumerate}

\begin{enumerate}
\item[$\mathbf{B}$:] $g_{i j}^{(0)} > 0$, $i,j \neq 0$.
\end{enumerate}

\begin{enumerate}
\item[$\mathbf{C}$:] There exists $\beta > 0$ such that:
\begin{enumerate}
\item[$\mathbf{(a)}$] $\limsup_{0 \leq \varepsilon \rightarrow 0} p_{i j}^{(\varepsilon)}(\beta) < \infty$, for all $i \neq 0$, $j \in X$.
\item[$\mathbf{(b)}$] $\phi_{i i}^{(0)}(\beta_i) \in (1,\infty)$, for some $i \neq 0$ and $\beta_i \leq \beta$.
\end{enumerate}
\end{enumerate}

It follows from conditions $\mathbf{A}$ and $\mathbf{B}$ that $\{ 1,\ldots,N \}$ is a communicating class of states for sufficiently small $\varepsilon$. Let us also remark that if $p_{i 0}^{(0)} = 0$ for all $i \neq 0$, it can be shown that part $\mathbf{(b)}$ of condition $\mathbf{C}$ always holds under conditions $\mathbf{A}$, $\mathbf{B}$, and $\mathbf{C} \mathbf{(a)}$.

\section{Systems of Linear Equations} \label{sec:systems}

In this section we derive systems of linear equations for exponential moment functionals.

We first consider the moment generating functions $\phi_{i j}^{(\varepsilon)}(\rho)$, defined by equation \eqref{eq:defphi}. By conditioning on $(\eta_1^{(\varepsilon)}, \kappa_1^{(\varepsilon)})$, we get for each $i,j \neq 0$,
\begin{equation} \label{eq:systemphipre}
\begin{split}
 \phi_{i j}^{(\varepsilon)}(\rho) &= \sum_{l \in X} \sum_{k=1}^\infty \mathsf{E}_i ( e^{\rho \mu_j^{(\varepsilon)}} \chi( \nu_0^{(\varepsilon)} > \nu_j^{(\varepsilon)}) | \eta_1^{(\varepsilon)} = l, \ \kappa_1^{(\varepsilon)} = k ) Q_{i l}^{(\varepsilon)}(k) \\
 &= \sum_{k=1}^\infty e^{\rho k} Q_{i j}^{(\varepsilon)}(k) + \sum_{l \neq 0,j} \sum_{k=1}^\infty \mathsf{E}_l e^{\rho (k + \mu_j^{(\varepsilon)})} \chi( \nu_0^{(\varepsilon)} > \nu_j^{(\varepsilon)}) Q_{i l}^{(\varepsilon)}(k).
\end{split}
\end{equation}

Relation \eqref{eq:systemphipre} gives us the following system of linear equations,
\begin{equation} \label{eq:systemphi}
 \phi_{i j}^{(\varepsilon)}(\rho) = p_{i j}^{(\varepsilon)}(\rho) + \sum_{l \neq 0,j} p_{i l}^{(\varepsilon)}(\rho) \phi_{l j}^{(\varepsilon)}(\rho), \ i,j \neq 0.
\end{equation}

In what follows it will often be convenient to use matrix notation. Let us introduce the following column vectors,
\begin{equation} \label{eq:vectorphi}
 \mathsf{\Phi}_j^{(\varepsilon)}(\rho) =
 \begin{bmatrix}
  \phi_{1 j}^{(\varepsilon)}(\rho) & \cdots & \phi_{N j}^{(\varepsilon)}(\rho)
 \end{bmatrix}
 ^T, \ j \neq 0,
\end{equation}
\begin{equation} \label{eq:vectorp}
 \mathbf{p}_j^{(\varepsilon)}(\rho) =
 \begin{bmatrix}
  p_{1 j}^{(\varepsilon)}(\rho) & \cdots & p_{N j}^{(\varepsilon)}(\rho)
 \end{bmatrix}
 ^T, \ j \in X.
\end{equation}

For each $j \neq 0$, we also define $N \times N$-matrices ${_{j}\mathbf{P}}^{(\varepsilon)}(\rho) = \| {_{j}p}_{i k}^{(\varepsilon)}(\rho) \|$ where the elements are given by
\begin{equation} \label{eq:matrixjP}
 {_{j}p}_{i k}^{(\varepsilon)}(\rho) = \left\{
 \begin{array}{l l}
  p_{i k}^{(\varepsilon)}(\rho) & i=1,\ldots,N, \ k \neq j, \\
  0 & i=1,\ldots,N, \ k = j.
 \end{array}
 \right.
\end{equation}

Using \eqref{eq:vectorphi}, \eqref{eq:vectorp}, and \eqref{eq:matrixjP}, we can write the system \eqref{eq:systemphi} in the following matrix form,
\begin{equation} \label{eq:systemphimatrix}
 \mathsf{\Phi}_j^{(\varepsilon)}(\rho) = \mathbf{p}_j^{(\varepsilon)}(\rho) + {_{j}\mathbf{P}}^{(\varepsilon)}(\rho) \mathsf{\Phi}_j^{(\varepsilon)}(\rho), \ j \neq 0.
\end{equation}

Note that the relations given above hold for all $\rho \in \mathbb{R}$ even in the case where some of the quantities involved take the value infinity. In this case we use the convention $0 \cdot \infty = 0$ and the equalities may take the form $\infty = \infty$.

Let us now derive a similar type of system for the following exponential moment functionals,
\begin{equation*}
 \omega_{i j s}^{(\varepsilon)}(\rho) = \sum_{n=0}^\infty e^{\rho n} \mathsf{P}_i \{ \xi^{(\varepsilon)}(n) = s, \ \mu_0^{(\varepsilon)} \wedge \mu_j^{(\varepsilon)} > n \}, \ \rho \in \mathbb{R}, \ i,j,s \in X.
\end{equation*}

First, note that
\begin{equation*}
\begin{split}
 \omega_{i j s}^{(\varepsilon)}(\rho) &= \mathsf{E}_i \sum_{n=0}^\infty e^{\rho n} \chi( \xi^{(\varepsilon)}(n) = s, \ \mu_0^{(\varepsilon)} \wedge \mu_j^{(\varepsilon)} > n ) \\
 &= \mathsf{E}_i \sum_{n=0}^{\mu_0^{(\varepsilon)} \wedge \mu_j^{(\varepsilon)} - 1} e^{\rho n} \chi( \xi^{(\varepsilon)}(n) = s ).
\end{split}
\end{equation*}

We now decompose $\omega_{i j s}^{(\varepsilon)}(\rho)$ into two parts,
\begin{equation} \label{eq:omegadecomp}
 \omega_{i j s}^{(\varepsilon)}(\rho) = \mathsf{E}_i \sum_{n=0}^{\kappa_1^{(\varepsilon)} - 1} e^{\rho n} \chi( \xi^{(\varepsilon)}(n) = s ) + \mathsf{E}_i \sum_{n = \kappa_1^{(\varepsilon)}}^{\mu_0^{(\varepsilon)} \wedge \mu_j^{(\varepsilon)} - 1} e^{\rho n} \chi( \xi^{(\varepsilon)}(n) = s ).
\end{equation}

Let us first rewrite the first term on the right hand side of equation \eqref{eq:omegadecomp}. By conditioning on $\kappa_1^{(\varepsilon)}$ we get, for $i,s \neq 0$,
\begin{equation*}
\begin{split}
 &\mathsf{E}_i \sum_{n=0}^{\kappa_1^{(\varepsilon)} - 1} e^{\rho n} \chi( \xi^{(\varepsilon)}(n) = s ) \\
 &\quad \quad = \sum_{k=1}^\infty \mathsf{E}_i \left( \sum_{n=0}^{\kappa_1^{(\varepsilon)} - 1} e^{\rho n} \chi( \xi^{(\varepsilon)}(n) = s ) \, \Big| \, \kappa_1^{(\varepsilon)} = k \right) \mathsf{P}_i \{ \kappa_1^{(\varepsilon)} = k \} \\
 &\quad \quad = \sum_{k=1}^\infty \delta(i,s) \left( \sum_{n=0}^{k - 1} e^{\rho n} \right) \mathsf{P}_i \{ \kappa_1^{(\varepsilon)} = k \}.
\end{split}
\end{equation*}
It follows that
\begin{equation} \label{eq:firstterm}
 \mathsf{E}_i \sum_{n=0}^{\kappa_1^{(\varepsilon)} - 1} e^{\rho n} \chi( \xi^{(\varepsilon)}(n) = s ) = \delta(i,s) \varphi_i^{(\varepsilon)}(\rho), \ i,s \neq 0.
\end{equation}
where
\begin{equation} \label{eq:defvarphi}
 \varphi_i^{(\varepsilon)}(\rho) = \left\{
 \begin{array}{l l}
  \mathsf{E}_i \kappa_1^{(\varepsilon)} & \rho = 0, \\
  ( \mathsf{E}_i e^{\rho \kappa_1^{(\varepsilon)}} - 1 ) / (e^\rho - 1) & \rho \neq 0.
 \end{array}
 \right.
\end{equation}

Let us now consider the second term on the right hand side of equation \eqref{eq:omegadecomp}. By conditioning on $(\eta_1^{(\varepsilon)},\kappa_1^{(\varepsilon)})$ we get, for $i,j,s \neq 0$,
\begin{equation*}
\begin{split}
 &\mathsf{E}_i \sum_{n = \kappa_1^{(\varepsilon)}}^{\mu_0^{(\varepsilon)} \wedge \mu_j^{(\varepsilon)} - 1} e^{\rho n} \chi( \xi^{(\varepsilon)}(n) = s ) \\
 &\quad \quad = \sum_{l \neq 0,j} \sum_{k=1}^\infty \mathsf{E}_i \left( \sum_{n = \kappa_1^{(\varepsilon)}}^{\mu_0^{(\varepsilon)} \wedge \mu_j^{(\varepsilon)} - 1} e^{\rho n} \chi( \xi^{(\varepsilon)}(n) = s ) \, \Big| \, \eta_1^{(\varepsilon)} = l, \ \kappa_1^{(\varepsilon)} = k \right) Q_{i l}^{(\varepsilon)}(k) \\
 &\quad \quad = \sum_{l \neq 0,j} \sum_{k=1}^\infty \mathsf{E}_l \left( \sum_{n = 0}^{\mu_0^{(\varepsilon)} \wedge \mu_j^{(\varepsilon)} - 1} e^{\rho (k + n)} \chi( \xi^{(\varepsilon)}(n) = s ) \right) Q_{i l}^{(\varepsilon)}(k).
\end{split}
\end{equation*}
It follows that
\begin{equation} \label{eq:secondterm}
 \mathsf{E}_i \sum_{n = \kappa_1^{(\varepsilon)}}^{\mu_0^{(\varepsilon)} \wedge \mu_j^{(\varepsilon)} - 1} e^{\rho n} \chi( \xi^{(\varepsilon)}(n) = s ) = \sum_{l \neq 0,j} p_{i l}^{(\varepsilon)}(\rho) \omega_{l j s}^{(\varepsilon)}(\rho), \ i,j,s \neq 0.
\end{equation}

From \eqref{eq:omegadecomp}, \eqref{eq:firstterm}, and \eqref{eq:secondterm} we now get the following system of linear equations,
\begin{equation} \label{eq:systemomega}
 \omega_{i j s}^{(\varepsilon)}(\rho) = \delta(i,s) \varphi_i^{(\varepsilon)}(\rho) + \sum_{l \neq 0,j} p_{i l}^{(\varepsilon)}(\rho) \omega_{l j s}^{(\varepsilon)}(\rho), \ i,j,s \neq 0.
\end{equation}

In order to write this system in matrix form, let us define the following column vectors,
\begin{equation} \label{eq:vectorvarphi}
 \widehat{\boldsymbol{\varphi}}_s^{(\varepsilon)}(\rho) =
 \begin{bmatrix}
  \delta(1,s) \varphi_1^{(\varepsilon)}(\rho) & \cdots & \delta(N,s) \varphi_N^{(\varepsilon)}(\rho)
 \end{bmatrix}
 ^T, \ s \neq 0,
\end{equation}
\begin{equation} \label{eq:vectoromega}
 \boldsymbol{\omega}_{j s}^{(\varepsilon)}(\rho) =
 \begin{bmatrix}
  \omega_{1 j s}^{(\varepsilon)}(\rho) & \cdots & \omega_{N j s}^{(\varepsilon)}(\rho)
 \end{bmatrix}
 ^T, \ j,s \neq 0.
\end{equation}

Using \eqref{eq:matrixjP}, \eqref{eq:vectorvarphi}, and \eqref{eq:vectoromega}, the system \eqref{eq:systemomega} can be written in the following matrix form,
\begin{equation} \label{eq:systemomegamatrix}
 \boldsymbol{\omega}_{j s}^{(\varepsilon)}(\rho) = \widehat{\boldsymbol{\varphi}}_s^{(\varepsilon)}(\rho) + {_{j}\mathbf{P}}^{(\varepsilon)}(\rho) \boldsymbol{\omega}_{j s}^{(\varepsilon)}(\rho), \ j,s \neq 0.
\end{equation}

We close this section with a lemma which will be important in what follows.

\begin{lemma} \label{lmm:finite}
Assume that we for some $\varepsilon \geq 0$ and $\rho \in \mathbb{R}$ have that $g_{i k}^{(\varepsilon)} > 0$, $i,k \neq 0$ and $p_{i k}^{(\varepsilon)}(\rho) < \infty$, $i \neq 0$, $k \in X$. Then, for any $j \neq 0$, the following statements are equivalent:
\begin{enumerate}
\item[$\mathbf{(a)}$] $\mathsf{\Phi}_j^{(\varepsilon)}(\rho) < \infty$.
\item[$\mathbf{(b)}$] $\boldsymbol{\omega}_{j s}^{(\varepsilon)}(\rho) < \infty$, $s \neq 0$.
\item[$\mathbf{(c)}$] The inverse matrix $( \mathbf{I} - {_{j}\mathbf{P}}^{(\varepsilon)}(\rho) )^{-1}$ exists.
\end{enumerate}
\end{lemma}

\begin{proof}
For each $j \neq 0$, let us define a matrix valued function ${_{j}\mathbf{A}}^{(\varepsilon)}(\rho) = \| {_{j}a}_{i k}^{(\varepsilon)}(\rho) \|$ by the relation
\begin{equation} \label{eq:matrixjA}
 {_{j}\mathbf{A}}^{(\varepsilon)}(\rho) = \mathbf{I} + {_{j}\mathbf{P}}^{(\varepsilon)}(\rho) + ({_{j}\mathbf{P}}^{(\varepsilon)}(\rho))^2 + \cdots, \ \rho \in \mathbb{R}.
\end{equation}

Since each term on the right hand side of \eqref{eq:matrixjA} is non-negative, it follows that the elements ${_{j}a}_{i k}^{(\varepsilon)}(\rho)$ are well defined and take values in the set $[0,\infty]$. Furthermore, the elements can be written in the following form which gives a probabilistic interpretation,
\begin{equation} \label{eq:probinterpret}
 {_{j}a}_{i k}^{(\varepsilon)}(\rho) = \mathsf{E}_i \sum_{n=0}^\infty e^{\rho \tau^{(\varepsilon)}(n)} \chi( \nu_0^{(\varepsilon)} \wedge \nu_j^{(\varepsilon)} > n, \ \eta_n^{(\varepsilon)} = k ), \ i,k \neq 0.
\end{equation}

Let us now show that
\begin{equation} \label{eq:matrixphialt}
 \mathsf{\Phi}_j^{(\varepsilon)}(\rho) = {_{j}\mathbf{A}}^{(\varepsilon)}(\rho) \mathbf{p}_j^{(\varepsilon)}(\rho), \ \rho \in \mathbb{R}, \ j \neq 0.
\end{equation}
In order to do this, first note that, for $j \neq 0$,
\begin{equation} \label{eq:indicator}
 \chi( \nu_0^{(\varepsilon)} > \nu_j^{(\varepsilon)} ) = \sum_{n=0}^\infty \sum_{k \neq 0} \chi( \nu_0^{(\varepsilon)} \wedge \nu_j^{(\varepsilon)} > n, \ \eta_n^{(\varepsilon)} = k, \ \eta_{n+1}^{(\varepsilon)} = j ).
\end{equation}

Using \eqref{eq:indicator} and the regenerative property of the semi-Markov process, the following is obtained, for $i,j \neq 0$,
\begin{equation} \label{eq:phialtpre}
\begin{split}
 \phi_{i j}^{(\varepsilon)}(\rho) &= \sum_{n=0}^\infty \sum_{k \neq 0} \mathsf{E}_i e^{\rho \mu_j^{(\varepsilon)}} \chi( \nu_0^{(\varepsilon)} \wedge \nu_j^{(\varepsilon)} > n, \ \eta_n^{(\varepsilon)} = k, \ \eta_{n+1}^{(\varepsilon)} = j ) \\
 &= \sum_{n=0}^\infty \sum_{k \neq 0} \mathsf{E}_i e^{\rho \tau^{(\varepsilon)}(n)} \chi( \nu_0^{(\varepsilon)} \wedge \nu_j^{(\varepsilon)} > n, \ \eta_n^{(\varepsilon)} = k ) p_{k j}^{(\varepsilon)}(\rho).
\end{split}
\end{equation}

From \eqref{eq:probinterpret} and \eqref{eq:phialtpre} we get
\begin{equation} \label{eq:phialt}
 \phi_{i j}^{(\varepsilon)}(\rho) = \sum_{k \neq 0} {_{j}a}_{i k}^{(\varepsilon)}(\rho) p_{k j}^{(\varepsilon)}(\rho), \ i,j \neq 0,
\end{equation}
and this proves \eqref{eq:matrixphialt}.

Let us now define
\begin{equation} \label{eq:omegasum}
 \omega_{i j}^{(\varepsilon)}(\rho) = \sum_{s \neq 0} \omega_{i j s}^{(\varepsilon)}(\rho) = \sum_{n=0}^\infty e^{\rho n} \mathsf{P}_i \{ \mu_0^{(\varepsilon)} \wedge \mu_j^{(\varepsilon)} > n \}, \ \rho \in \mathbb{R}, \ \ i,j \neq 0.
\end{equation}

Then, we have
\begin{equation} \label{eq:omegasum2}
 \omega_{i j}^{(\varepsilon)}(\rho) = \left\{
 \begin{array}{l l}
  \mathsf{E}_i (\mu_0^{(\varepsilon)} \wedge \mu_j^{(\varepsilon)}) & \rho = 0, \\
  ( \mathsf{E}_i e^{\rho (\mu_0^{(\varepsilon)} \wedge \mu_j^{(\varepsilon)})} - 1) / (e^\rho - 1) & \rho \neq 0.
 \end{array}
 \right.
\end{equation}

Also notice that
\begin{equation} \label{eq:expdecomp}
 \mathsf{E}_i e^{\rho (\mu_0^{(\varepsilon)} \wedge \mu_j^{(\varepsilon)})} = \mathsf{E}_i e^{\rho \mu_j^{(\varepsilon)}} \chi( \nu_0^{(\varepsilon)} > \nu_j^{(\varepsilon)} ) + \mathsf{E}_i e^{\rho \mu_0^{(\varepsilon)}} \chi( \nu_0^{(\varepsilon)} < \nu_j^{(\varepsilon)} ), \ i,j \neq 0.
\end{equation}

Using similar calculations as above, it can be shown that
\begin{equation} \label{eq:phitildealt}
 \mathsf{E}_i e^{\rho \mu_0^{(\varepsilon)}} \chi( \nu_0^{(\varepsilon)} < \nu_j^{(\varepsilon)} ) = \sum_{k \neq 0} {_{j}a}_{i k}^{(\varepsilon)}(\rho) p_{k 0}^{(\varepsilon)}(\rho), \ i,j \neq 0.
\end{equation}

It follows from \eqref{eq:phialt}, \eqref{eq:expdecomp}, and \eqref{eq:phitildealt} that
\begin{equation} \label{eq:expdecomp2}
 \mathsf{E}_i e^{\rho (\mu_0^{(\varepsilon)} \wedge \mu_j^{(\varepsilon)})} = \sum_{k \neq 0} {_{j}a}_{i k}^{(\varepsilon)}(\rho) \left( p_{k j}^{(\varepsilon)}(\rho) + p_{k 0}^{(\varepsilon)}(\rho) \right), \ i,j \neq 0.
\end{equation}

Let us now show that $\mathbf{(a)}$ implies $\mathbf{(b)}$.

By iterating relation \eqref{eq:systemphimatrix} we obtain,
\begin{equation} \label{eq:phiiterated}
\begin{split}
 \mathsf{\Phi}_j^{(\varepsilon)}(\rho) &= ( \mathbf{I} + {_{j}\mathbf{P}}^{(\varepsilon)}(\rho) + \cdots + ({_{j}\mathbf{P}}^{(\varepsilon)}(\rho))^n ) \mathbf{p}_j^{(\varepsilon)}(\rho) \\
 &+ ({_{j}\mathbf{P}}^{(\varepsilon)}(\rho))^{n+1} \mathsf{\Phi}_j^{(\varepsilon)}(\rho), \ n=1,2,\ldots
\end{split}
\end{equation}

Since $\mathsf{\Phi}_j^{(\varepsilon)}(\rho) < \infty$, it follows from \eqref{eq:phiiterated} that
\begin{equation} \label{eq:phiiterated2}
 ({_{j}\mathbf{P}}^{(\varepsilon)}(\rho))^{n+1} \mathsf{\Phi}_j^{(\varepsilon)}(\rho) \rightarrow \mathbf{0}, \ \text{as} \ n \rightarrow \infty.
\end{equation}

The assumptions of the lemma guarantee that $\mathsf{\Phi}_j^{(\varepsilon)}(\rho) > 0$. From this and relation \eqref{eq:phiiterated2} we can conclude that $({_{j}\mathbf{P}}^{(\varepsilon)}(\rho))^{n+1} \rightarrow \mathbf{0}$, as $n \rightarrow \infty$. It is known that this holds if and only if the matrix series \eqref{eq:matrixjA} converges in norms, that is, ${_{j}\mathbf{A}}^{(\varepsilon)}(\rho)$ is finite. From this and relations \eqref{eq:omegasum}, \eqref{eq:omegasum2}, and \eqref{eq:expdecomp2} it follows that $\mathbf{(b)}$ holds.

Next we show that $\mathbf{(b)}$ implies $\mathbf{(c)}$.

By summing over all $s \neq 0$ in relation \eqref{eq:systemomegamatrix} it follows that
\begin{equation} \label{eq:systemomegasum}
 \boldsymbol{\omega}_j^{(\varepsilon)}(\rho) = \boldsymbol{\varphi}^{(\varepsilon)}(\rho) + {_{j}\mathbf{P}}^{(\varepsilon)}(\rho) \boldsymbol{\omega}_j^{(\varepsilon)}(\rho), \ \rho \in \mathbb{R},
\end{equation}
where
\begin{equation*}
 \boldsymbol{\omega}_j^{(\varepsilon)}(\rho) =
 \begin{bmatrix}
  \omega_{1 j}^{(\varepsilon)}(\rho) & \cdots & \omega_{N j}^{(\varepsilon)}(\rho)
 \end{bmatrix}
 ^T, \ j \neq 0,
\end{equation*}
and
\begin{equation*}
 \boldsymbol{\varphi}^{(\varepsilon)}(\rho) =
 \begin{bmatrix}
  \varphi_1^{(\varepsilon)}(\rho) & \cdots & \varphi_N^{(\varepsilon)}(\rho)
 \end{bmatrix}
 ^T.
\end{equation*}

By iterating relation \eqref{eq:systemomegasum} we get
\begin{equation} \label{eq:omegaiterated}
\begin{split}
 \boldsymbol{\omega}_j^{(\varepsilon)}(\rho) &= ( \mathbf{I} + {_{j}\mathbf{P}}^{(\varepsilon)}(\rho) + \cdots + ({_{j}\mathbf{P}}^{(\varepsilon)}(\rho))^n ) \boldsymbol{\varphi}^{(\varepsilon)}(\rho) \\
 &+ ({_{j}\mathbf{P}}^{(\varepsilon)}(\rho))^{n+1} \boldsymbol{\omega}_j^{(\varepsilon)}(\rho), \ n=1,2,\ldots
\end{split}
\end{equation}

It follows from $\mathbf{(b)}$ and the definition of $\omega_{i j}^{(\varepsilon)}(\rho)$ that $0 < \boldsymbol{\omega}_j^{(\varepsilon)}(\rho) < \infty$. So, letting $n \rightarrow \infty$ in \eqref{eq:omegaiterated} and using similar arguments as above, it follows that the matrix series \eqref{eq:matrixjA} converges in norms. It is then known that the inverse matrix $(\mathbf{I} - {_{j}\mathbf{P}}^{(\varepsilon)}(\rho) )^{-1}$ exists, that is, $\mathbf{(c)}$ holds.

Let us finally argue that $\mathbf{(c)}$ implies $\mathbf{(a)}$.

If $(\mathbf{I} - {_{j}\mathbf{P}}^{(\varepsilon)}(\rho) )^{-1}$ exists, then the following relation holds,
\begin{equation} \label{eq:inverserel}
 ( \mathbf{I} - {_{j}\mathbf{P}}^{(\varepsilon)}(\rho) )^{-1} = \mathbf{I} + {_{j}\mathbf{P}}^{(\varepsilon)}(\rho) ( \mathbf{I} - {_{j}\mathbf{P}}^{(\varepsilon)}(\rho) )^{-1}.
\end{equation}

Iteration of \eqref{eq:inverserel} gives
\begin{equation} \label{eq:inverserelit}
\begin{split}
 ( \mathbf{I} - {_{j}\mathbf{P}}^{(\varepsilon)}(\rho) )^{-1} &= \mathbf{I} + {_{j}\mathbf{P}}^{(\varepsilon)}(\rho) + ({_{j}\mathbf{P}}^{(\varepsilon)}(\rho))^2 + \cdots + ({_{j}\mathbf{P}}^{(\varepsilon)}(\rho))^n \\
 &+ ({_{j}\mathbf{P}}^{(\varepsilon)}(\rho))^{n+1} ( \mathbf{I} - {_{j}\mathbf{P}}^{(\varepsilon)}(\rho) )^{-1}, \ n=1,2,\ldots
\end{split}
\end{equation}
Letting $n \rightarrow \infty$ in \eqref{eq:inverserelit} it follows that ${_{j}\mathbf{A}}^{(\varepsilon)}(\rho) = ( \mathbf{I} - {_{j}\mathbf{P}}^{(\varepsilon)}(\rho) )^{-1} < \infty$. From \eqref{eq:matrixphialt} we now see that $\mathbf{(a)}$ holds.
\end{proof}

\section{Convergence of Moment Functionals} \label{sec:convergence}

In this section it is shown that the solutions of the systems derived in Section \ref{sec:systems} converge as the perturbation parameter tends to zero. In addition, we prove some properties for the solution of a characteristic equation.

Let us define
\begin{equation*}
 {_{k}\phi}_{i j}^{(\varepsilon)}(\rho) = \mathsf{E}_i e^{\rho \mu_j^{(\varepsilon)}} \chi( \nu_0^{(\varepsilon)} \wedge \nu_k^{(\varepsilon)} > \nu_j^{(\varepsilon)} ), \ \rho \in \mathbb{R}, \ i,j,k \in X.
\end{equation*}

If the states $\{ 1,\ldots,N \}$ is a communicating class and $\phi_{i i}^{(\varepsilon)}(\rho) \leq 1$ for some $i \neq 0$, then it can be shown (see, for example, Petersson (2015)) that the following relation holds for all $j \neq 0$,
\begin{equation} \label{eq:solidarityrel}
 ( 1 - \phi_{i i}^{(\varepsilon)}(\rho) )( 1 - {_{i}\phi}_{j j}^{(\varepsilon)}(\rho) ) = ( 1 - \phi_{j j}^{(\varepsilon)}(\rho) )( 1 - {_{j}\phi}_{i i}^{(\varepsilon)}(\rho) ).
\end{equation}

Relation \eqref{eq:solidarityrel} is useful in order to prove various solidarity properties for semi-Markov processes. In particular, if $\phi_{i i}^{(\varepsilon)}(\rho) = 1$, relation \eqref{eq:solidarityrel} reduces to
\begin{equation} \label{eq:solidarityrel2}
 ( 1 - \phi_{j j}^{(\varepsilon)}(\rho) )( 1 - {_{j}\phi}_{i i}^{(\varepsilon)}(\rho) ) = 0.
\end{equation}

From the regenerative property of the semi-Markov process it follows that
\begin{equation} \label{eq:solidarityrel3}
 \phi_{i i}^{(\varepsilon)}(\rho) = {_{j}\phi}_{i i}^{(\varepsilon)}(\rho) + {_{i}\phi}_{i j}^{(\varepsilon)}(\rho) \phi_{j i}^{(\varepsilon)}(\rho), \ j \neq 0,i.
\end{equation}

Since $\{ 1,\ldots,N \}$ is a communicating class, we have ${_{i}\phi}_{i j}^{(\varepsilon)}(\rho) > 0$ and $\phi_{j i}^{(\varepsilon)}(\rho) > 0$. So, if $\phi_{i i}^{(\varepsilon)}(\rho) = 1$ it follows from \eqref{eq:solidarityrel3} that ${_{j}\phi}_{i i}^{(\varepsilon)}(\rho) < 1$. From this and \eqref{eq:solidarityrel2} we can conclude that $\phi_{j j}^{(\varepsilon)}(\rho) = 1$ for all $j \neq 0$. Thus, we have the following lemma:

\begin{lemma} \label{lmm:solidarity}
Assume that we for some $\varepsilon \geq 0$ have that $g_{k j}^{(\varepsilon)} > 0$ for all $k,j \neq 0$. Then, if we for some $i \neq 0$ and $\rho \in \mathbb{R}$, have that $\phi_{i i}^{(\varepsilon)}(\rho) = 1$, it follows that $\phi_{j j}^{(\varepsilon)}(\rho) = 1$ for all $j \neq 0$.
\end{lemma}

Let us now define the following characteristic equation,
\begin{equation} \label{eq:chareq}
 \phi_{i i}^{(\varepsilon)}(\rho) = 1.
\end{equation}
where $i \neq 0$ is arbitrary. The root of equation \eqref{eq:chareq} plays an important role for the asymptotic behaviour of the corresponding semi-Markov process, see, for example, Petersson (2016).

The following lemma gives limits of moment functionals and properties for the root of the characteristic equation.

\begin{lemma} \label{lmm:delta}
If conditions $\mathbf{A}$--$\mathbf{C}$ hold, then there exists $\delta \in (0,\beta]$ such that the following holds:
\begin{enumerate}
\item[$\mathbf{(i)}$] $\phi_{k j}^{(\varepsilon)}(\rho) \rightarrow \phi_{k j}^{(0)}(\rho) < \infty$, as $\varepsilon \rightarrow 0$, $\rho \leq \delta$, $k,j \neq 0$.
\item[$\mathbf{(ii)}$] $\omega_{k j s}^{(\varepsilon)}(\rho) \rightarrow \omega_{k j s}^{(0)}(\rho) < \infty$, as $\varepsilon \rightarrow 0$, $\rho \leq \delta$, $k,j,s \neq 0$.
\item[$\mathbf{(iii)}$] $\phi_{j j}^{(0)}(\delta) \in (1,\infty)$, $j \neq 0$.
\item[$\mathbf{(iv)}$] For sufficiently small $\varepsilon$, there exists a unique non-negative root $\rho^{(\varepsilon)}$ of the characteristic equation \eqref{eq:chareq} which does not depend on $i$.
\item[$\mathbf{(v)}$] $\rho^{(\varepsilon)} \rightarrow \rho^{(0)} < \delta$ as $\varepsilon \rightarrow 0$.
\end{enumerate}
\end{lemma}

\begin{proof}
Let $i \neq 0$ and $\beta_i \leq \beta$ be the values given in condition $\mathbf{C}$. It follows from conditions $\mathbf{B}$ and $\mathbf{C}$ that $\phi_{i i}^{(0)}(\rho)$ is a continuous and strictly increasing function for $\rho \leq \beta_i$. Since $\phi_{i i}^{(0)}(0) = g_{i i}^{(0)} \leq 1$ and $\phi_{i i}^{(0)}(\beta_i) > 1$, there exists a unique $\rho' \in [0,\beta_i)$ such that $\phi_{i i}^{(0)}(\rho') = 1$. Moreover, by Lemma \ref{lmm:solidarity},
\begin{equation} \label{eq:lmmdelta1}
 \phi_{j j}^{(0)}(\rho') = 1, \ j \neq 0.
\end{equation}

For all $j \neq 0$, we have
\begin{equation} \label{eq:lmmdelta2}
\begin{split}
 \phi_{j j}^{(0)}(\rho') = {_{k}\phi}_{j j}^{(0)}(\rho') + {_{j}\phi}_{j k}^{(0)}(\rho') \phi_{k j}^{(0)}(\rho'), \ k \neq 0,j.
\end{split}
\end{equation}

It follows from \eqref{eq:lmmdelta1}, \eqref{eq:lmmdelta2}, and condition $\mathbf{B}$, that
\begin{equation} \label{eq:lmmdelta3}
 \phi_{k j}^{(0)}(\rho') < \infty, \ k,j \neq 0.
\end{equation}

From \eqref{eq:lmmdelta3} and Lemma \ref{lmm:finite} we get that $\det( \mathbf{I} - {_{j}\mathbf{P}}^{(0)}(\rho') ) \neq 0$, for $j \neq 0$. Under condition $\mathbf{C}$, the elements of $\mathbf{I} - {_{j}\mathbf{P}}^{(0)}(\rho)$ are continuous functions for $\rho \leq \beta$. This implies that we for each $j \neq 0$ can find $\beta_j \in (\rho',\beta_i]$ such that $\det( \mathbf{I} - {_{j}\mathbf{P}}^{(0)}(\beta_j) ) \neq 0$. By condition $\mathbf{C}$ we also have that $p_{k j}^{(0)}(\beta_j) < \infty$ for $k \neq 0$, $j \in X$. It now follows from Lemma \ref{lmm:finite} that $\phi_{k j}^{(0)}(\beta_j) < \infty$, $k,j \neq 0$. If we define $\delta = \min \{ \beta_1,\ldots,\beta_N \}$, it follows that
\begin{equation} \label{eq:lmmdelta4}
 \phi_{k j}^{(0)}(\rho) < \infty, \ \rho \leq \delta, \ k,j \neq 0.
\end{equation}

Now, let $\rho \leq \delta$ be fixed. Relation \eqref{eq:lmmdelta4} and Lemma \ref{lmm:finite} imply that
\begin{equation} \label{eq:lmmdelta5}
 \det( \mathbf{I} - {_{j}\mathbf{P}}^{(0)}(\rho) ) \neq 0, \ j \neq 0.
\end{equation}

Note that we have
\begin{equation} \label{eq:lmmdelta6}
 p_{k j}^{(\varepsilon)}(\rho) = p_{k j}^{(\varepsilon)} \sum_{n=0}^\infty e^{\rho n} f_{k j}^{(\varepsilon)}(n), \ k,j \in X.
\end{equation}
Since $f_{k j}^{(\varepsilon)}(n)$ are proper probability distributions, it follows from \eqref{eq:lmmdelta6} and conditions $\mathbf{A}$ and $\mathbf{C}$ that
\begin{equation} \label{eq:lmmdelta7}
 p_{k j}^{(\varepsilon)}(\rho) \rightarrow p_{k j}^{(0)}(\rho) < \infty, \ \text{as} \ \varepsilon \rightarrow 0, \ k \neq 0, \ j \in X.
\end{equation}

It follows from \eqref{eq:lmmdelta5} and \eqref{eq:lmmdelta7} that there exists $\varepsilon_1 > 0$ such that we for all $\varepsilon \leq \varepsilon_1$ have that $\det( \mathbf{I} - {_{j}\mathbf{P}}^{(\varepsilon)}(\rho) ) \neq 0$ and $p_{k j}^{(\varepsilon)}(\rho) < \infty$, for all $k,j \neq 0$. Using Lemma \ref{lmm:finite} once again, it now follows that $\phi_{k j}^{(\varepsilon)}(\rho) < \infty$, $k,j \neq 0$, for all $\varepsilon \leq \varepsilon_1$. Moreover, in this case, the system of linear equations \eqref{eq:systemphimatrix} has a unique solution for $\varepsilon \leq \varepsilon_1$ given by
\begin{equation} \label{eq:lmmdelta8}
 \mathsf{\Phi}_j^{(\varepsilon)}(\rho) = ( \mathbf{I} - {_{j}\mathbf{P}}^{(\varepsilon)}(\rho) )^{-1} \mathbf{p}_j^{(\varepsilon)}(\rho), \ j \neq 0.
\end{equation}

From \eqref{eq:lmmdelta7} and \eqref{eq:lmmdelta8} it follows that
\begin{equation*}
 \phi_{k j}^{(\varepsilon)}(\rho) \rightarrow \phi_{k j}^{(0)}(\rho) < \infty, \ \text{as} \ \varepsilon \rightarrow 0, \ k,j \neq 0.
\end{equation*}

This completes the proof of part $\mathbf{(i)}$.

For the proof of part $\mathbf{(ii)}$ we first note that, since $\phi_{k j}^{(\varepsilon)}(\rho) < \infty$ for $\varepsilon \leq \varepsilon_1$, $k,j \neq 0$, it follows from Lemma \ref{lmm:finite} that $\omega_{k j s}^{(\varepsilon)}(\rho) < \infty$ for $\varepsilon \leq \varepsilon_1$, $k,j,s \neq 0$. From this, and arguments given above, we see that the system of linear equations given by relation \eqref{eq:systemomegamatrix} has a unique solution for $\varepsilon \leq \varepsilon_1$ given by
\begin{equation} \label{eq:lmmdelta10}
 \boldsymbol{\omega}_{j s}^{(\varepsilon)}(\rho) = ( \mathbf{I} - {_{j}\mathbf{P}}^{(\varepsilon)}(\rho) )^{-1} \widehat{\boldsymbol{\varphi}}_s^{(\varepsilon)}(\rho), \ j,s \neq 0.
\end{equation}

Now, since $\mathsf{E}_i e^{\rho \kappa_1^{(\varepsilon)}} = \sum_{j \in X} p_{i j}^{(\varepsilon)}(\rho)$, it follows from \eqref{eq:defvarphi} and \eqref{eq:lmmdelta7} that $\varphi_i^{(\varepsilon)}(\rho) \rightarrow \varphi_i^{(0)}(\rho) < \infty$ as $\varepsilon \rightarrow 0$, $i \neq 0$. Using this and relations \eqref{eq:lmmdelta7} and \eqref{eq:lmmdelta10} we can conclude that part $\mathbf{(ii)}$ holds.

By part $\mathbf{(i)}$ we have, in particular, $\phi_{j j}^{(\varepsilon)}(\delta) \rightarrow \phi_{j j}^{(0)}(\delta) < \infty$ as $\varepsilon \rightarrow 0$, for all $j \neq 0$. Furthermore, since $\rho' < \delta$ and $\phi_{j j}^{(0)}(\rho)$ is strictly increasing for $\rho \leq \delta$, it follows from \eqref{eq:lmmdelta1} that $\phi_{j j}^{(0)}(\delta) > 1$, $j \neq 0$. This proves part $\mathbf{(iii)}$.

Let us now prove part $\mathbf{(iv)}$.

It follows from $\mathbf{(i)}$ and $\mathbf{(iii)}$ that we can find $\varepsilon_2 > 0$ such that $\phi_{j j}^{(\varepsilon)}(\delta) \in (1,\infty)$, $j \neq 0$, for all $\varepsilon \leq \varepsilon_2$. By conditions $\mathbf{A}$ and $\mathbf{B}$ there exists $\varepsilon_3 > 0$ such that, for each $i \neq 0$ and $\varepsilon \leq \varepsilon_3$, the function $g_{i i}^{(\varepsilon)}(n)$ is not concentrated at zero. Thus, for every $i \neq 0$ and $\varepsilon \leq \min \{ \varepsilon_2, \varepsilon_3 \}$, we have that $\phi_{i i}^{(\varepsilon)}(\rho)$ is a continuous and strictly increasing function for $\rho \in [0,\delta]$. Since $\phi_{i i}^{(\varepsilon)}(0) = g_{i i}^{(\varepsilon)} \leq 1$ and $\phi_{i i}^{(\varepsilon)}(\delta) > 1$, there exists a unique $\rho_i^{(\varepsilon)} \in [0,\delta)$ such that $\phi_{i i}^{(\varepsilon)}(\rho_i^{(\varepsilon)}) = 1$. By Lemma \ref{lmm:solidarity}, the root of the characteristic equation does not depend on $i$ so we can write $\rho^{(\varepsilon)}$ instead of $\rho_i^{(\varepsilon)}$. This proves part $\mathbf{(iv)}$.

Finally, we show that $\rho^{(\varepsilon)} \rightarrow \rho^{(0)}$ as $\varepsilon \rightarrow 0$.

Let $\gamma > 0$ such that $\rho^{(0)} + \gamma \leq \delta$ be arbitrary. Then $\phi_{i i}^{(0)}(\rho^{(0)} - \gamma) < 1$ and $\phi_{i i}^{(0)}(\rho^{(0)} + \gamma) > 1$. From this and part $\mathbf{(i)}$ we get that there exists $\varepsilon_4 > 0$ such that $\phi_{i i}^{(\varepsilon)}(\rho^{(0)} - \gamma) < 1$ and $\phi_{i i}^{(\varepsilon)}(\rho^{(0)} + \gamma) > 1$, for all $\varepsilon \leq \varepsilon_4$. So, it follows that $|\rho^{(\varepsilon)} - \rho^{(0)}| < \gamma$ for $\varepsilon \leq \min \{ \varepsilon_2, \varepsilon_3, \varepsilon_4 \}$. This completes the proof of Lemma \ref{lmm:delta}.
\end{proof}

\section{Expansions of Moment Functionals} \label{sec:expansions}

In this section, asymptotic expansions for mixed power-exponential moment functionals are constructed. The main results are given by Theorems \ref{thm:expansionphi} and \ref{thm:expansionomega}.

Let us define the following mixed power-exponential moment functionals for distributions of first hitting times,
\begin{equation*}
 \phi_{i j}^{(\varepsilon)}(\rho,r) = \sum_{n=0}^\infty n^r e^{\rho n} g_{i j}^{(\varepsilon)}(n), \ \rho \in \mathbb{R}, \ r=0,1,\ldots, \ i,j \in X.
\end{equation*}
By definition, $\phi_{i j}^{(\varepsilon)}(\rho,0) = \phi_{i j}^{(\varepsilon)}(\rho)$.

We also define the following mixed power-exponential moment functionals for transition probabilities,
\begin{equation*}
 p_{i j}^{(\varepsilon)}(\rho,r) = \sum_{n=0}^\infty n^r e^{\rho n} Q_{i j}^{(\varepsilon)}(n), \ \rho \in \mathbb{R}, \ r=0,1,\ldots, \ i,j \in X.
\end{equation*}
By definition, $p_{i j}^{(\varepsilon)}(\rho,0) = p_{i j}^{(\varepsilon)}(\rho)$.

It follows from conditions $\mathbf{A}$--$\mathbf{C}$ and Lemma \ref{lmm:delta} that, for $\rho < \delta$ and sufficiently small $\varepsilon$, the functions $\phi_{i j}^{(\varepsilon)}(\rho)$ and $p_{i j}^{(\varepsilon)}(\rho)$ are arbitrarily many times differentiable with respect to $\rho$, and the derivatives of order $r$ are given by $\phi_{i j}^{(\varepsilon)}(\rho,r)$ and $p_{i j}^{(\varepsilon)}(\rho,r)$, respectively.

Recall from Section \ref{sec:systems} that the following system of linear equations holds,
\begin{equation} \label{eq:systemphi0}
 \phi_{i j}^{(\varepsilon)}(\rho) = p_{i j}^{(\varepsilon)}(\rho) + \sum_{l \neq 0,j} p_{i l}^{(\varepsilon)}(\rho) \phi_{l j}^{(\varepsilon)}(\rho), \ i,j \neq 0.
\end{equation}
Differentiating relation \eqref{eq:systemphi0} gives
\begin{equation} \label{eq:systemphir}
 \phi_{i j}^{(\varepsilon)}(\rho,r) = \lambda_{i j}^{(\varepsilon)}(\rho,r) + \sum_{l \neq 0,j} p_{i l}^{(\varepsilon)}(\rho) \phi_{l j}^{(\varepsilon)}(\rho,r), \ r=1,2,\ldots, \ i,j \neq 0,
\end{equation}
where
\begin{equation} \label{eq:lambdar}
 \lambda_{i j}^{(\varepsilon)}(\rho,r) = p_{i j}^{(\varepsilon)}(\rho,r) + \sum_{m=1}^r \binom{r}{m} \sum_{l \neq 0,j} p_{i l}^{(\varepsilon)}(\rho,m) \phi_{l j}^{(\varepsilon)}(\rho,r-m).
\end{equation}

In order to write relations \eqref{eq:systemphi0}, \eqref{eq:systemphir}, and \eqref{eq:lambdar} in matrix form, let us define the following column vectors,
\begin{equation} \label{eq:vectorphir}
 \mathsf{\Phi}_j^{(\varepsilon)}(\rho,r) =
 \begin{bmatrix}
  \phi_{1 j}^{(\varepsilon)}(\rho,r) & \cdots & \phi_{N j}^{(\varepsilon)}(\rho,r)
 \end{bmatrix}
 ^T, \ j \neq 0,
\end{equation}
\begin{equation} \label{eq:vectorpr}
 \mathbf{p}_j^{(\varepsilon)}(\rho,r) =
 \begin{bmatrix}
  p_{1 j}^{(\varepsilon)}(\rho,r) & \cdots & p_{N j}^{(\varepsilon)}(\rho,r)
 \end{bmatrix}
 ^T, \ j \neq 0,
\end{equation}
\begin{equation} \label{eq:vectorlambdar}
 \boldsymbol{\lambda}_j^{(\varepsilon)}(\rho,r) =
 \begin{bmatrix}
  \lambda_{1 j}^{(\varepsilon)}(\rho,r) & \cdots & \lambda_{N j}^{(\varepsilon)}(\rho,r)
 \end{bmatrix}
 ^T, \ j \neq 0.
\end{equation}

Let us also, for $j \neq 0$, define $N \times N$-matrices ${_{j}\mathbf{P}}^{(\varepsilon)}(\rho, r) = \| {_{j}p}_{i k}^{(\varepsilon)}(\rho,r) \|$ where the elements are given by
\begin{equation} \label{eq:matrixjPr}
 {_{j}p}_{i k}^{(\varepsilon)}(\rho, r) = \left\{
 \begin{array}{l l}
  p_{i k}^{(\varepsilon)}(\rho,r) & i=1,\ldots,N, \ k \neq j, \\
  0 & i=1,\ldots,N, \ k = j. 
 \end{array}
 \right.
\end{equation}

Using \eqref{eq:systemphi0}--\eqref{eq:matrixjPr} we can for any $j \neq 0$ write the following recursive systems of linear equations,
\begin{equation} \label{eq:systemphi0matrix}
 \mathsf{\Phi}_j^{(\varepsilon)}(\rho) = \mathbf{p}_j^{(\varepsilon)}(\rho) + {_{j}\mathbf{P}}^{(\varepsilon)}(\rho) \mathsf{\Phi}_j^{(\varepsilon)}(\rho),
\end{equation}
and, for $r = 1,2,\ldots,$
\begin{equation} \label{eq:systemphirmatrix}
 \mathsf{\Phi}_j^{(\varepsilon)}(\rho,r) = \boldsymbol{\lambda}_j^{(\varepsilon)}(\rho,r) + {_{j}\mathbf{P}}^{(\varepsilon)}(\rho) \mathsf{\Phi}_j^{(\varepsilon)}(\rho,r),
\end{equation}
where
\begin{equation} \label{eq:systemlambdarmatrix}
 \boldsymbol{\lambda}_j^{(\varepsilon)}(\rho,r) = \mathbf{p}_j^{(\varepsilon)}(\rho,r) + \sum_{m=1}^r \binom{r}{m} {_{j}\mathbf{P}}^{(\varepsilon)}(\rho,m) \mathsf{\Phi}_j^{(\varepsilon)}(\rho,r-m).
\end{equation}

Let us now introduce the following perturbation condition, which is assumed to hold for some $\rho < \delta$, where $\delta$ is the parameter in Lemma \ref{lmm:delta}:
\begin{enumerate}
\item[$\mathbf{P_k^*}$:] $p_{i j}^{(\varepsilon)}(\rho,r) = p_{i j}^{(0)}(\rho,r) + p_{i j}[\rho,r,1] \varepsilon + \cdots + p_{i j}[\rho,r,k-r] \varepsilon^{k-r} + o(\varepsilon^{k-r})$, for $r=0,\ldots,k$, $i \neq 0$, $j \in X$, where $|p_{i j}[\rho,r,n]| < \infty$, for $r=0,\ldots,k$, $n=1,\ldots,k-r$, $i \neq 0$, $j \in X$.
\end{enumerate}
For convenience, we denote $p_{i j}^{(0)}(\rho,r) = p_{i j}[\rho,r,0]$, for $r=0,\ldots,k$.

Note that if condition $\mathbf{P_k^*}$ holds, then, for $r=0,\ldots,k$, we have the following asymptotic matrix expansions,
\begin{equation} \label{eq:expansionjPr}
 {_{j}\mathbf{P}}^{(\varepsilon)}(\rho,r) = {_{j}\mathbf{P}}[\rho,r,0] + {_{j}\mathbf{P}}[\rho,r,1] \varepsilon + \cdots + {_{j}\mathbf{P}}[\rho,r,k-r] \varepsilon^{k-r} + \mathbf{o}(\varepsilon^{k-r}),
\end{equation}
\begin{equation} \label{eq:expansionpr}
 \mathbf{p}_j^{(\varepsilon)}(\rho,r) = \mathbf{p}_j[\rho,r,0] + \mathbf{p}_j[\rho,r,1] \varepsilon + \cdots + \mathbf{p}_j[\rho,r,k-r] \varepsilon^{k-r} + \mathbf{o}(\varepsilon^{k-r}).
\end{equation}
Here, and in what follows, $\mathbf{o}(\varepsilon^p)$ denotes a matrix-valued function of $\varepsilon$ where all elements are of order $o(\varepsilon^p)$. The coefficients in \eqref{eq:expansionjPr} are $N \times N$-matrices ${_{j}\mathbf{P}}[\rho,r,n] = \| {_{j}p}_{i k}[\rho,r,n] \|$ with elements given by
\begin{equation*}
 {_{j}p}_{i k}[\rho,r,n] = \left\{
 \begin{array}{l l}
  p_{i k}[\rho,r,n] & i=1,\ldots,N, \ k \neq j, \\
  0 & i=1,\ldots,N, \ k=j,
 \end{array}
 \right.
\end{equation*}
and the coefficients in \eqref{eq:expansionpr} are column vectors defined by
\begin{equation*}
 \mathbf{p}_j[\rho,r,n] =
 \begin{bmatrix}
  p_{1 j}[\rho,r,n] & \cdots & p_{N j}[\rho,r,n]
 \end{bmatrix}
 ^T.
\end{equation*}

Let us now define the following matrix, which will play an important role in what follows,
\begin{equation*}
 {_{j}\mathbf{U}}^{(\varepsilon)}(\rho) = ( \mathbf{I} - {_{j}\mathbf{P}}^{(\varepsilon)}(\rho) )^{-1}.
\end{equation*}
Under conditions $\mathbf{A}$--$\mathbf{C}$, it follows from Lemmas \ref{lmm:finite} and \ref{lmm:delta} that ${_{j}\mathbf{U}}^{(\varepsilon)}(\rho)$ is well defined for $\rho \leq \delta$ and sufficiently small $\varepsilon$.

The following lemma gives an asymptotic expansion for ${_{j}\mathbf{U}}^{(\varepsilon)}(\rho)$.
\begin{lemma} \label{lmm:expansionjU}
Assume that conditions $\mathbf{A}$--$\mathbf{C}$ and $\mathbf{P_k^*}$ hold. Then we have the following asymptotic expansion,
\begin{equation} \label{eq:expansionjU}
 {_{j}\mathbf{U}}^{(\varepsilon)}(\rho) = {_{j}\mathbf{U}}[\rho,0] + {_{j}\mathbf{U}}[\rho,1] \varepsilon + \cdots + {_{j}\mathbf{U}}[\rho,k] \varepsilon^k + \mathbf{o}(\varepsilon^k),
\end{equation}
where
\begin{equation} \label{eq:coeffjU}
 {_{j}\mathbf{U}}[\rho,n] = \left\{
 \begin{array}{l l}
  ( \mathbf{I} - {_{j}\mathbf{P}}^{(0)}(\rho) )^{-1} & n=0, \\
  {_{j}\mathbf{U}}[\rho,0] \sum_{q=1}^n {_{j}\mathbf{P}}[\rho,0,q] {_{j}\mathbf{U}}[\rho,n-q] & n=1,\ldots,k.
 \end{array}
 \right.
\end{equation}
\end{lemma}

\begin{proof}
As already mentioned above, conditions $\mathbf{A}$--$\mathbf{C}$ ensure us that the inverse ${_{j}\mathbf{U}}^{(\varepsilon)}(\rho)$ exists for sufficiently small $\varepsilon$. In this case, it is known that the expansion \eqref{eq:expansionjU} exists under condition $\mathbf{P_k^*}$. To see that the coefficients are given by \eqref{eq:coeffjU}, first note that
\begin{equation} \label{eq:lmmexpansionjU}
\begin{split}
 \mathbf{I} &= ( \mathbf{I} - {_{j}\mathbf{P}}^{(\varepsilon)}(\rho) ) {_{j}\mathbf{U}}^{(\varepsilon)}(\rho) \\
 &= ( \mathbf{I} - {_{j}\mathbf{P}}^{(0)}(\rho) - {_{j}\mathbf{P}}[\rho,0,1] \varepsilon - \cdots - {_{j}\mathbf{P}}[\rho,0,k] \varepsilon^k + \mathbf{o}(\varepsilon^k) ) \\
 &\times ( {_{j}\mathbf{U}}[\rho,0] + {_{j}\mathbf{U}}[\rho,1] \varepsilon + \cdots + {_{j}\mathbf{U}}[\rho,k] \varepsilon^k + \mathbf{o}(\varepsilon^k) ).
\end{split}
\end{equation}
By first expanding both sides of equation \eqref{eq:lmmexpansionjU} and then, for $n=0,1,\ldots,k$, equating coefficients of $\varepsilon^n$ in the left and right hand sides, we get formula \eqref{eq:coeffjU}.
\end{proof}

We are now ready to construct asymptotic expansions for $\mathsf{\Phi}_j^{(\varepsilon)}(\rho,r)$.

\begin{theorem} \label{thm:expansionphi}
Assume that conditions $\mathbf{A}$--$\mathbf{C}$ and $\mathbf{P_k^*}$ hold. Then:
\begin{enumerate}
\item[$\mathbf{(i)}$] We have the following asymptotic expansion,
\begin{equation*}
 \mathsf{\Phi}_j^{(\varepsilon)}(\rho) = \mathsf{\Phi}_j[\rho,0,0] + \mathsf{\Phi}_j[\rho,0,1] \varepsilon + \cdots + \mathsf{\Phi}_j[\rho,0,k] \varepsilon^k + \mathbf{o}(\varepsilon^k),
\end{equation*}
where
\begin{equation*}
 \mathsf{\Phi}_j[\rho,0,n] = \left\{
 \begin{array}{l l}
  \mathsf{\Phi}_j^{(0)}(\rho) & n=0, \\
  \sum_{q=0}^n {_{j}\mathbf{U}}[\rho,q] \mathbf{p}_j[\rho,0,n-q] & n=1,\ldots,k.
 \end{array}
 \right.
\end{equation*}
\item[$\mathbf{(ii)}$] For $r=1,\ldots,k$, we have the following asymptotic expansions,
\begin{equation*}
 \mathsf{\Phi}_j^{(\varepsilon)}(\rho,r) = \mathsf{\Phi}_j[\rho,r,0] + \mathsf{\Phi}_j[\rho,r,1] \varepsilon + \cdots + \mathsf{\Phi}_j[\rho,r,k-r] \varepsilon^{k-r} + \mathbf{o}(\varepsilon^{k-r}),
\end{equation*}
where
\begin{equation*}
 \mathsf{\Phi}_j[\rho,r,n] = \left\{
 \begin{array}{l l}
  \mathsf{\Phi}_j^{(0)}(\rho,r) & n=0, \\
  \sum_{q=0}^n {_{j}\mathbf{U}}[\rho,q] \boldsymbol{\lambda}_j[\rho,r,n-q] & n=1,\ldots,k-r,
 \end{array}
 \right.
\end{equation*}
and, for $t=0,\ldots,k-r$,
\begin{equation*}
 \boldsymbol{\lambda}_j[\rho,r,t] = \mathbf{p}_j[\rho,r,t] + \sum_{m=1}^r \binom{r}{m} \sum_{q=0}^t {_{j}\mathbf{P}}[\rho,m,q] \mathsf{\Phi}_j[\rho,r-m,t-q].
\end{equation*}
\end{enumerate}
\end{theorem}

Before proceeding with the proof of Theorem \ref{thm:expansionphi} we would like to comment on the reason that the theorem is stated in such a way that $\mathsf{\Phi}_j^{(\varepsilon)}(\rho,r)$, for $r=1,\ldots,k$, has an expansion of order $k-r$. The reason is that this is exactly what we need for the main result in Petersson (2016), which, we remind, is a sequel of the present paper. However, it is possible to construct asymptotic expansions of different orders than the ones stated in the theorem. In that case, appropriate changes in the perturbation condition should be made. The same remark applies to Lemma \ref{lmm:expansionvarphi} and Theorem \ref{thm:expansionomega} below.

\begin{proof}
Under conditions $\mathbf{A}$--$\mathbf{C}$, we have, for sufficiently small $\varepsilon$, that the recursive systems of linear equations given by relations \eqref{eq:systemphi0matrix}, \eqref{eq:systemphirmatrix}, and \eqref{eq:systemlambdarmatrix}, all have finite components. Moreover, the inverse matrix ${_{j}\mathbf{U}}^{(\varepsilon)}(\rho) = ( \mathbf{I} - {_{j}\mathbf{P}}^{(\varepsilon)}(\rho) )^{-1}$ exists, so these systems have unique solutions.

It follows from \eqref{eq:systemphi0matrix}, Lemma \ref{lmm:expansionjU}, and condition $\mathbf{P_k^*}$ that
\begin{equation} \label{eq:lmmexpansionphi1}
\begin{split}
 \mathsf{\Phi}_j^{(\varepsilon)}(\rho) &= {_{j}\mathbf{U}}^{(\varepsilon)}(\rho) \mathbf{p}_j^{(\varepsilon)}(\rho) \\
 &= ( {_{j}\mathbf{U}}[\rho,0] + {_{j}\mathbf{U}}[\rho,1] \varepsilon + \cdots + {_{j}\mathbf{U}}[\rho,k] \varepsilon^k + \mathbf{o}(\varepsilon^k) ) \\
 &\times ( \mathbf{p}_j[\rho,0,0] + \mathbf{p}_j[\rho,0,1] \varepsilon + \cdots + \mathbf{p}_j[\rho,0,k] \varepsilon^k + \mathbf{o}(\varepsilon^k) ).
\end{split}
\end{equation}

By expanding the right hand side of equation \eqref{eq:lmmexpansionphi1}, we see that part $\mathbf{(i)}$ of Theorem \ref{thm:expansionphi} holds.

With $r=1$, relation \eqref{eq:systemlambdarmatrix} takes the form
\begin{equation} \label{eq:lmmexpansionphi2}
 \boldsymbol{\lambda}_j^{(\varepsilon)}(\rho,1) = \mathbf{p}_j^{(\varepsilon)}(\rho,1) + {_{j}\mathbf{P}}^{(\varepsilon)}(\rho,1) \mathsf{\Phi}_j^{(\varepsilon)}(\rho).
\end{equation}

From \eqref{eq:lmmexpansionphi2}, condition $\mathbf{P_k^*}$, and part $\mathbf{(i)}$, we get
\begin{equation} \label{eq:lmmexpansionphi3}
\begin{split}
 \boldsymbol{\lambda}_j^{(\varepsilon)}(\rho,1) &= \mathbf{p}_j[\rho,1,0] + \cdots + \mathbf{p}_j[\rho,1,k-1] \varepsilon^{k-1} + \mathbf{o}(\varepsilon^{k-1}) \\
 &+ ( {_{j}\mathbf{P}}[\rho,1,0] + \cdots + {_{j}\mathbf{P}}[\rho,1,k-1] \varepsilon^{k-1} + \mathbf{o}(\varepsilon^{k-1}) ) \\
 &\times ( \mathsf{\Phi}_j[\rho,0,0] + \cdots + \mathsf{\Phi}_j[\rho,0,k-1] \varepsilon^{k-1} + \mathbf{o}(\varepsilon^{k-1}) ).
\end{split}
\end{equation}

Expanding the right hand side of \eqref{eq:lmmexpansionphi3} gives
\begin{equation} \label{eq:lmmexpansionphi4}
 \boldsymbol{\lambda}_j^{(\varepsilon)}(\rho,1) = \boldsymbol{\lambda}_j[\rho,1,0] + \boldsymbol{\lambda}_j[\rho,1,1] \varepsilon + \cdots + \boldsymbol{\lambda}_j[\rho,1,k-1] \varepsilon^{k-1} + \mathbf{o}(\varepsilon^{k-1}),
\end{equation}
where
\begin{equation*}
 \boldsymbol{\lambda}_j[\rho,1,t] = \mathbf{p}_j[\rho,1,t] + \sum_{q=0}^t {_{j}\mathbf{P}}[\rho,1,q] \mathsf{\Phi}_j[\rho,0,t-q], \ t=0,\ldots,k-1.
\end{equation*}

It now follows from \eqref{eq:systemphirmatrix}, \eqref{eq:lmmexpansionphi4}, and Lemma \ref{lmm:expansionjU} that
\begin{equation} \label{eq:lmmexpansionphi5}
\begin{split}
 \mathsf{\Phi}_j^{(\varepsilon)}(\rho,1) &= {_{j}\mathbf{U}}^{(\varepsilon)}(\rho) \boldsymbol{\lambda}_j^{(\varepsilon)}(\rho,1) \\
 &= ( {_{j}\mathbf{U}}[\rho,0] + \cdots + {_{j}\mathbf{U}}[\rho,k-1] \varepsilon^{k-1} + \mathbf{o}(\varepsilon^{k-1}) ) \\
 &\times ( \boldsymbol{\lambda}_j[\rho,1,0] + \cdots + \boldsymbol{\lambda}_j[\rho,1,k-1] \varepsilon^{k-1} + \mathbf{o}(\varepsilon^{k-1}) ).
\end{split}
\end{equation}

By expanding the right hand side of equation \eqref{eq:lmmexpansionphi5} we get the expansion in part $\mathbf{(ii)}$ for $r=1$. If $k=1$, this concludes the proof. If $k \geq 2$, we can repeat the steps above, successively, for $r=2,\ldots,k$. This gives the expansions and formulas given in part $\mathbf{(ii)}$.
\end{proof}

Let us now define the following mixed power exponential moment functionals, for $i,j,s \in X$,
\begin{equation*}
 \omega_{i j s}^{(\varepsilon)}(\rho,r) = \sum_{n=0}^\infty n^r e^{\rho n} \mathsf{P}_i \{ \xi^{(\varepsilon)}(n) = s, \ \mu_0^{(\varepsilon)} \wedge \mu_j^{(\varepsilon)} > n \}, \ \rho \in \mathbb{R}, \ r=0,1,\ldots
\end{equation*}
Notice that $\omega_{i j s}^{(\varepsilon)}(\rho,0) = \omega_{i j s}^{(\varepsilon)}(\rho)$.

It follows from conditions $\mathbf{A}$--$\mathbf{C}$ and Lemma \ref{lmm:delta} that for $\rho < \delta$ and sufficiently small $\varepsilon$, the functions $\omega_{i j s}^{(\varepsilon)}(\rho)$ and $p_{i j}^{(\varepsilon)}(\rho)$ are arbitrarily many times differentiable with respect to $\rho$, and the derivatives of order $r$ are given by $\omega_{i j s}^{(\varepsilon)}(\rho,r)$ and $p_{i j}^{(\varepsilon)}(\rho,r)$, respectively. Under these conditions we also have that the functions $\varphi_i^{(\varepsilon)}(\rho)$, defined by equation \eqref{eq:defvarphi}, are differentiable. Let us denote the corresponding derivatives by $\varphi_i^{(\varepsilon)}(\rho,r)$.

Recall from Section \ref{sec:systems} that the functions $\omega_{i j s}^{(\varepsilon)}(\rho)$ satisfy the following system of linear equations,
\begin{equation} \label{eq:systemomega0}
 \omega_{i j s}^{(\varepsilon)}(\rho) = \delta(i,s) \varphi_i^{(\varepsilon)}(\rho) + \sum_{l \neq 0,j} p_{i l}^{(\varepsilon)}(\rho) \omega_{l j s}^{(\varepsilon)}(\rho), \ i,j,s \neq 0.
\end{equation}

Differentiating relation \eqref{eq:systemomega0} gives
\begin{equation} \label{eq:systemomegar}
 \omega_{i j s}^{(\varepsilon)}(\rho,r) = \theta_{i j s}^{(\varepsilon)}(\rho,r) + \sum_{l \neq 0,j}  p_{i l}^{(\varepsilon)}(\rho) \omega_{l j s}^{(\varepsilon)}(\rho,r), \ r=1,2,\ldots, \ i,j,s \neq 0,
\end{equation}
where
\begin{equation} \label{eq:thetar}
 \theta_{i j s}^{(\varepsilon)}(\rho,r) = \delta(i,s) \varphi_i^{(\varepsilon)}(\rho,r) + \sum_{m=1}^r \binom{r}{m} \sum_{l \neq 0,j} p_{i l}^{(\varepsilon)}(\rho,m) \omega_{l j s}^{(\varepsilon)}(\rho,r-m).
\end{equation}

In order to rewrite these systems in matrix form, we define the following column vectors,
\begin{equation} \label{eq:vectoromegar}
 \boldsymbol{\omega}_{j s}^{(\varepsilon)}(\rho,r) =
 \begin{bmatrix}
  \omega_{1 j s}^{(\varepsilon)}(\rho,r) & \cdots & \omega_{N j s}^{(\varepsilon)}(\rho,r)
 \end{bmatrix}
 ^T, \ j,s \neq 0,
\end{equation}
\begin{equation} \label{eq:vectorthetar}
 \boldsymbol{\theta}_{j s}^{(\varepsilon)}(\rho,r) =
 \begin{bmatrix}
  \theta_{1 j s}^{(\varepsilon)}(\rho,r) & \cdots & \theta_{N j s}^{(\varepsilon)}(\rho,r)
 \end{bmatrix}
 ^T, \ j,s \neq 0,
\end{equation}
\begin{equation} \label{eq:vectorvarphir}
 \widehat{\boldsymbol{\varphi}}_s^{(\varepsilon)}(\rho,r) =
 \begin{bmatrix}
  \delta(1,s) \varphi_1^{(\varepsilon)}(\rho,r) & \cdots & \delta(N,s) \varphi_N^{(\varepsilon)}(\rho,r)
 \end{bmatrix}
 ^T, \ s \neq 0.
\end{equation}

Using \eqref{eq:matrixjPr} and \eqref{eq:systemomega0}--\eqref{eq:vectorvarphir}, we can for each $j,s \neq 0$ write the following recursive systems of linear equations,
\begin{equation} \label{eq:systemomega0matrix}
 \boldsymbol{\omega}_{j s}^{(\varepsilon)}(\rho) = \widehat{\boldsymbol{\varphi}}_s^{(\varepsilon)}(\rho) + {_{j}\mathbf{P}}^{(\varepsilon)}(\rho) \boldsymbol{\omega}_{j s}^{(\varepsilon)}(\rho),
\end{equation}
and, for $r=1,2,\ldots,$
\begin{equation} \label{eq:systemomegarmatrix}
 \boldsymbol{\omega}_{j s}^{(\varepsilon)}(\rho,r) = \boldsymbol{\theta}_{j s}^{(\varepsilon)}(\rho,r) + {_{j}\mathbf{P}}^{(\varepsilon)}(\rho) \boldsymbol{\omega}_{j s}^{(\varepsilon)}(\rho,r),
\end{equation}
where
\begin{equation} \label{eq:systemthetarmatrix}
 \boldsymbol{\theta}_{j s}^{(\varepsilon)}(\rho,r) = \widehat{\boldsymbol{\varphi}}_s^{(\varepsilon)}(\rho,r) + \sum_{m=1}^r \binom{r}{m} {_{j}\mathbf{P}}^{(\varepsilon)}(\rho,m) \boldsymbol{\omega}_{j s}^{(\varepsilon)}(\rho,r-m).
\end{equation}

In order to construct asymptotic expansions for the vectors $\boldsymbol{\omega}_{j s}^{(\varepsilon)}(\rho,r)$, we can use the same technique as in Theorem \ref{thm:expansionphi}. However, a preliminary step needed in this case is to construct asymptotic expansions for the functions $\varphi_i^{(\varepsilon)}(\rho,r)$. In order to do this, we first derive an expression for these functions.

Let us define
\begin{equation} \label{eq:defpsir2}
 \psi_i^{(\varepsilon)}(\rho,r) = \sum_{n=0}^\infty n^r e^{\rho n} \mathsf{P}_i \{ \kappa_1^{(\varepsilon)} = n \}, \ \rho \in \mathbb{R}, \ r=0,1,\ldots, \ i \in X.
\end{equation}

Note that
\begin{equation} \label{eq:psirrel}
 \psi_i^{(\varepsilon)}(\rho,r) = \sum_{j \in X} p_{i j}^{(\varepsilon)}(\rho,r), \ \rho \in \mathbb{R}, \ r=0,1,\ldots, \ i \in X.
\end{equation}
Thus, the functions $\psi_i^{(\varepsilon)}(\rho,0)$ are arbitrarily many times differentiable with respect to $\rho$ and the corresponding derivatives are given by $\psi_i^{(\varepsilon)}(\rho,r)$.

The function $\varphi_i^{(\varepsilon)}(\rho)$, defined by equation \eqref{eq:defvarphi}, can be written as
\begin{equation} \label{eq:defvarphi2}
 \varphi_i^{(\varepsilon)}(\rho) = \left\{
 \begin{array}{l l}
  \psi_i^{(\varepsilon)}(0,1) & \rho = 0, \\
  ( \psi_i^{(\varepsilon)}(\rho,0) - 1 ) / (e^\rho - 1) & \rho \neq 0.
 \end{array}
 \right.
\end{equation}

From \eqref{eq:defpsir2} and \eqref{eq:defvarphi2} it follows that
\begin{equation} \label{eq:varphirel}
 \psi_i^{(\varepsilon)}(\rho,0) = (e^\rho - 1) \varphi_i^{(\varepsilon)}(\rho) + 1, \ \rho \in \mathbb{R}.
\end{equation}

Differentiating both sides of \eqref{eq:varphirel} gives
\begin{equation} \label{eq:varphirel2}
 \psi_i^{(\varepsilon)}(\rho,r) = (e^\rho - 1) \varphi_i^{(\varepsilon)}(\rho,r) + e^\rho \sum_{m=0}^{r-1} \binom{r}{m} \varphi_i^{(\varepsilon)}(\rho,m), \ r=1,2,\ldots
\end{equation}

If $\rho = 0$, equation \eqref{eq:varphirel2} implies
\begin{equation*}
 \psi_i^{(\varepsilon)}(0,r) = r \varphi_i^{(\varepsilon)}(0,r-1) + \sum_{m=0}^{r-2} \binom{r}{m} \varphi_i^{(\varepsilon)}(0,m), \ r=2,3,\ldots
\end{equation*}
From this it follows that, for $r=1,2,\ldots,$
\begin{equation} \label{eq:varphidiffzero}
 \varphi_i^{(\varepsilon)}(0,r) = \frac{1}{r+1} \left( \psi_i^{(\varepsilon)}(0,r+1) - \sum_{m=0}^{r-1} \binom{r+1}{m} \varphi_i^{(\varepsilon)}(0,m) \right).
\end{equation}

If $\rho \neq 0$, equation \eqref{eq:varphirel2} gives, for $r=1,2,\ldots,$
\begin{equation} \label{eq:varphidiff}
 \varphi_i^{(\varepsilon)}(\rho,r) = \frac{1}{e^\rho - 1} \left( \psi_i^{(\varepsilon)}(\rho,r) - e^\rho \sum_{m=0}^{r-1} \binom{r}{m} \varphi_i^{(\varepsilon)}(\rho,m) \right).
\end{equation}

Using relations \eqref{eq:psirrel}, \eqref{eq:varphidiffzero}, and \eqref{eq:varphidiff}, we can recursively calculate the derivatives of $\varphi_i^{(\varepsilon)}(\rho)$. Furthermore, it follows directly from these formulas that we can construct asymptotic expansions for these derivatives. The formulas are given in the following lemma.

\begin{lemma} \label{lmm:expansionvarphi}
Assume that conditions $\mathbf{A}$--$\mathbf{C}$ hold.
\begin{enumerate}
\item[$\mathbf{(i)}$] If, in addition, condition $\mathbf{P_k^*}$ holds, then for each $i \neq 0$ and $r=0,\ldots,k$ we have the following asymptotic expansion,
\begin{equation*}
 \psi_i^{(\varepsilon)}(\rho,r) = \psi_i[\rho,r,0] + \psi_i[\rho,r,1] \varepsilon + \cdots + \psi_i[\rho,r,k-r] \varepsilon^{k-r} + o(\varepsilon^{k-r}),
\end{equation*}
where
\begin{equation*}
 \psi_i[\rho,r,n] = \sum_{j \in X} p_{i j}[\rho,r,n], \ n=0,\ldots,k-r.
\end{equation*}
\item[$\mathbf{(ii)}$] If, in addition, $\rho = 0$ and condition $\mathbf{P_{k+1}^*}$ holds, then for each $i \neq 0$ and $r=0,\ldots,k$ we have the following asymptotic expansion,
\begin{equation*}
 \varphi_i^{(\varepsilon)}(0,r) = \varphi_i[0,r,0] + \varphi_i[0,r,1] \varepsilon + \cdots + \varphi_i[0,r,k-r] \varepsilon^{k-r} + o(\varepsilon^{k-r}),
\end{equation*}
where, for $n=0,\ldots,k-r$,
\begin{equation*}
 \varphi_i[0,r,n] = \frac{1}{r+1} \left( \psi_i[0,r+1,n] - \sum_{m=0}^{r-1} \binom{r+1}{m} \varphi_i[0,m,n] \right).
\end{equation*}
\item[$\mathbf{(iii)}$] If, in addition, $\rho \neq 0$ and condition $\mathbf{P_k^*}$ holds, then for each $i \neq 0$ and $r=0,\ldots,k$ we have the following asymptotic expansion,
\begin{equation*}
 \varphi_i^{(\varepsilon)}(\rho,r) = \varphi_i[\rho,r,0] + \varphi_i[\rho,r,1] \varepsilon + \cdots + \varphi_i[\rho,r,k-r] \varepsilon^{k-r} + o(\varepsilon^{k-r}),
\end{equation*}
where, for $n=0,\ldots,k-r$,
\begin{equation*}
 \varphi_i[\rho,r,n] = \frac{1}{e^\rho - 1} \left( \psi_i[\rho,r,n] - e^\rho \sum_{m=0}^{r-1} \binom{r}{m} \varphi_i[\rho,m,n] \right).
\end{equation*}
\end{enumerate}
\end{lemma}

Using \eqref{eq:vectorvarphir} and Lemma \ref{lmm:expansionvarphi} we can now construct the following asymptotic expansions, for $r=0,\ldots,k$, and $s \neq 0$,
\begin{equation} \label{eq:expansionvarphihat}
 \widehat{\boldsymbol{\varphi}}_s^{(\varepsilon)}(\rho,r) = \widehat{\boldsymbol{\varphi}}_s[\rho,r,0] + \widehat{\boldsymbol{\varphi}}_s[\rho,r,1] \varepsilon + \cdots + \widehat{\boldsymbol{\varphi}}_s[\rho,r,k-r] \varepsilon^{k-r} + \mathbf{o}(\varepsilon^{k-r}).
\end{equation}

The next lemma gives asymptotic expansions for $\boldsymbol{\omega}_{j s}^{(\varepsilon)}(\rho,r)$.

\begin{theorem} \label{thm:expansionomega}
Assume that conditions $\mathbf{A}$--$\mathbf{C}$ hold. If $\rho = 0$, we also assume that condition $\mathbf{P_{k+1}^*}$ holds. If $\rho \neq 0$, we also assume that condition $\mathbf{P_k^*}$ holds. Then:
\begin{enumerate}
\item[$\mathbf{(i)}$] We have the following asymptotic expansion,
\begin{equation*}
 \boldsymbol{\omega}_{j s}^{(\varepsilon)}(\rho) = \boldsymbol{\omega}_{j s}[\rho,0,0] + \boldsymbol{\omega}_{j s}[\rho,0,1] \varepsilon + \cdots + \boldsymbol{\omega}_{j s}[\rho,0,k] \varepsilon^k + \mathbf{o}(\varepsilon^k),
\end{equation*}
where
\begin{equation*}
 \boldsymbol{\omega}_{j s}[\rho,0,n] = \left\{
 \begin{array}{l l}
  \boldsymbol{\omega}_{j s}^{(0)}(\rho) & n=0, \\
  \sum_{q=0}^n {_{j}\mathbf{U}}[\rho,q] \widehat{\boldsymbol{\varphi}}_s[\rho,0,n-q] & n=1,\ldots,k.
 \end{array}
 \right.
\end{equation*}
\item[$\mathbf{(ii)}$] For $r=1,\ldots,k$, we have the following asymptotic expansions,
\begin{equation*}
 \boldsymbol{\omega}_{j s}^{(\varepsilon)}(\rho,r) = \boldsymbol{\omega}_{j s}[\rho,r,0] + \boldsymbol{\omega}_{j s}[\rho,r,1] \varepsilon + \cdots + \boldsymbol{\omega}_{j s}[\rho,r,k-r] \varepsilon^{k-r} + \mathbf{o}(\varepsilon^{k-r}),
\end{equation*}
where
\begin{equation*}
 \boldsymbol{\omega}_{j s}[\rho,r,n] = \left\{
 \begin{array}{l l}
  \boldsymbol{\omega}_{j s}^{(0)}(\rho,r) & n=0, \\
  \sum_{q=0}^n {_{j}\mathbf{U}}[\rho,q] \boldsymbol{\theta}_{j s}[\rho,r,n-q] & n=1,\ldots,k-r,
 \end{array}
 \right.
\end{equation*}
and, for $t=0,\ldots,k-r$,
\begin{equation*}
 \boldsymbol{\theta}_{j s}[\rho,r,t] = \widehat{\boldsymbol{\varphi}}_s[\rho,r,t] + \sum_{m=1}^r \binom{r}{m} \sum_{q=0}^t {_{j}\mathbf{P}}[\rho,m,q] \boldsymbol{\omega}_{j s}[\rho,r-m,t-q].
\end{equation*}
\end{enumerate}
\end{theorem}

\begin{proof}
Under conditions $\mathbf{A}$--$\mathbf{C}$, we have, for sufficiently small $\varepsilon$, that the recursive systems of linear equations given by relations \eqref{eq:systemomega0matrix}, \eqref{eq:systemomegarmatrix}, and \eqref{eq:systemthetarmatrix}, all have finite components. Moreover, the inverse matrix ${_{j}\mathbf{U}}^{(\varepsilon)}(\rho) = ( \mathbf{I} - {_{j}\mathbf{P}}^{(\varepsilon)}(\rho) )^{-1}$ exists, so these systems have unique solutions. Since we, by Lemma \ref{lmm:expansionvarphi}, have the expansions given in equation \eqref{eq:expansionvarphihat}, the proof is from this point analogous to the proof of Theorem \ref{thm:expansionphi}.
\end{proof}

% Bibliography

\end{document}